\documentclass[a4,12pt]{amsart}
\usepackage{amssymb,amstext,amsmath,amscd,amsthm,amsfonts,enumerate,graphicx,latexsym,color}
\usepackage[all]{xy}
\usepackage{cite}
\newtheorem{thm}{Theorem}[section]
\newtheorem*{thm*}{Theorem}

\newtheorem{cor}[thm]{Corollary}
\newtheorem{lem}[thm]{Lemma}

\theoremstyle{definition}
\newtheorem{dfn}[thm]{Definition}
\newtheorem*{dfn*}{Definition}
\newtheorem{rem}[thm]{Remark}

\newtheorem{ex}[thm]{Example}

\theoremstyle{remark}

\newtheorem*{claim*}{Claim}
\numberwithin{equation}{thm}
\def\m{\mathfrak{m}}
\def\p{\mathfrak{p}}
\def\h{\mathfrak{ht}}
\def\X{\mathcal{X}}

\def\om{\omega}
\def\d{\delta}
\def\H{\mathsf{H}}
\def\tr{\mathsf{Tr}}
\def\mn{\mathsf{min}}
\def\Ext{\mathsf{Ext}}
\def\End{\mathsf{End}}
\def\rad{\mathsf{rad}}
\def\Hom{\mathsf{Hom}}
\def\mod{\mathsf{mod}}
\def\cm{\mathsf{MCM}}
\def\syz{\mathsf{\Omega}}
\def\dim{\operatorname{\mathsf{dim}}}
\def\depth{\operatorname{\mathsf{depth}}}
\def\ass{\operatorname{\mathsf{Ass}}}
\def\add{\operatorname{\mathsf{add}}}
\def\supp{\operatorname{\mathsf{Supp}}}
\def\cod{\operatorname{\mathsf{codim}}}
\def\depth{\operatorname{\mathsf{depth}}}
\def\h{\operatorname{\mathsf{ht}}}
\def\im{\operatorname{\mathsf{Im}}}
\def\ker{\operatorname{\mathsf{Ker}}}
\def\cok{\operatorname{\mathsf{Coker}}}
\title[The structure of MCM-preenvelopes]{The structure of preenvelopes with respect to maximal Cohen-Macaulay modules}
\author{Hiroki Matsui} 
\address{Graduate School of Mathematics, Nagoya University, Furocho, Chikusaku, Nagoya, Aichi 464-8602, Japan}
\email{m14037f@math.nagoya-u.ac.jp}
\date{\today}
\begin{document}
\begin{abstract}
This paper studies the structure of special preenvelopes and envelopes with respect to maximal Cohen-Macaulay modules.
We investigate the structure of them in terms of their kernels and cokernels.
Moreover, using this result, we also study the structure of special proper coresolutions with respect to maximal Cohen-Macaulay modules over a Henselian Cohen-Macaulay local ring.
\end{abstract}
\maketitle
\section{Introduction}
Throughout this paper, we assume that $(R, \m, k)$ is a $d$-dimensional Cohen-Macaulay local ring with canonical module $\om$. 
All $R$-modules are assumed to be finitely generated. 
Denote by $\mod{R}$ the category of finitely generated $R$-modules and by $\cm$ the full subcategory of
$\mod{R}$ consisting of maximal Cohen-Macaulay $R$-modules. 

We define $(-)^\dagger:=\Hom_R(-, \om)$ and $\d_M:M \to M^{\dagger\dagger}$ as a natural homomorphism for an $R$-module $M$.
Note that if $M$ is maximal Cohen-Macaulay, $\d_M$ is an isomorphism.

Let $\X$ be a full subcategory of $\mod{R}$. 
The notion of $\X$-(pre)covers and $\X$-(pre)envelopes have been playing an important role in the representation theory of algebras; see \cite{AR, AS1, ASS, E, EJ, R3} for instance.  
For $\X=\cm$, a celebrated theorem due to Auslander and Buchweitz \cite{AB} says that for any $R$-module $M$, there exists a short exact sequence 
$$
0 \to Y \xrightarrow{f} X \xrightarrow{\pi} M \to 0
$$
where $X$ is maximal Cohen-Macaulay and $Y$ has finite injective dimension.
The map $\pi$ is called a maximal Cohen-Macaulay approximation of $M$.
Then $\pi$ is an $\cm$-precover of $M$, and is an $\cm$-cover if $Y$ and $X$ have no non-zero common direct summand via $f$. %$f \in \rad(Y, X)$
If $R$ is Henselian, every $R$-module has an $\cm$-cover; see \cite{R, Y}.

In this paper we mainly study the $\cm$-envelope, and the $\cm$-preenvelope which is called special.
%The existence of special $\cm$-preenvelopes and $\cm$-envelopes is shown by Holm \cite{H}.
A result of Holm \cite[Theorem A]{H} states that every $R$-module has a special $\cm$-preenvelope, and if $R$ is Henselian, every $R$-module has an $\cm$-envelope.
%\begin{thm}[Holm]\label{holm} 
%\begin{enumerate}[\rm(1)]
%\item
%Let $\pi: X \to M$ be an $R$-homomorphism such that $X$ is maximal Cohen-Macaulay. If $\pi$ is an $\cm$-precover (resp. a special $\cm$-precover; resp. an $\cm$-cover), then $\pi^\dagger \circ \d_M$ is an
%$\cm$-preenvelope (resp. a special $\cm$-preenvelope; resp. an $\cm$-envelope).
%\item
%Every $R$-module admits a special $\cm$-preenvelope.
%\item
%If $R$ is Henselian, then every $R$-module admits an $\cm$-envelope.
%\end{enumerate}
%\end{thm}
It is natural to ask when a given homomorphism is a special $\cm$-preenvelope or an $\cm$-envelope,
and we give an answer to this question by using the kernels and cokernels. 
Our first main result is the following theorem.
\begin{thm}\label{main1}
Let $\mu: M \to X$ be an $R$-homomorphism such that $X$ is maximal Cohen-Macaulay.
\begin{enumerate}[\rm(1)]
\item The following are equivalent.
	\begin{enumerate}[\rm(a)]
\item $\mu$ is a special $\cm$-preenvelope of $M$.
	\item $\cod(\ker{\mu}) >0$ and $\Ext_R^1(\cok{\mu}, \cm) =0 $.
	\item $\cod(\ker{\mu}) > 0$, and there exists an exact sequence 
	$
	0 \to S \to \cok{\mu} \to T \to U \to 0
	$
	such that
	\begin{itemize}
	\item $\cod{S} > 1$,
	\item $\cod{U} > 2$, 
	\item $T$ satisfies $(S_2)$,
	\item $T^\dagger$ has finite injective dimension and satisfies $(S_3)$.
	\end{itemize}
	\end{enumerate}
\item The following are equivalent if $R$ is Henselian.
	\begin{enumerate}[\rm(a)]
\item $\mu$ is an $\cm$-envelope of $M$.
	\item $\cod(\ker{\mu}) >0$, $\Ext_R^1(\cok{\mu}, \cm) =0 $ and $\cok(\mu^{\dagger\dagger})$ has no free summand.
	\item $\cod(\ker{\mu}) > 0$, and there exists an exact sequence 
	$
	0 \to S \to \cok{\mu} \xrightarrow{u} T \to U \to 0
	$
	such that
	\begin{itemize}
	\item $\cod{S} > 1$,
	\item $\cod{U} > 2$,
	\item $T$ satisfies $(S_2)$, 
	\item $T^\dagger$ has finite injective dimension and satisfies $(S_3)$, 
	\item $\im{u}$ has no non-zero free summand.
	\end{itemize}
	\end{enumerate}
\end{enumerate}
\end{thm}
The conditions (c) in (1) and (2) not only clarify the structure of special $\cm$-preenvelopes and $\cm$-envelopes, but also have the advantage that they do not contain vanishing conditions of Ext modules, which are in general hard to verify. 
We construct concrete examples of $\cm$-(pre)envelopes by using Theorem \ref{main1}.
Moreover, applying this result, we give another characterization of special $\cm$-preenvelopes in terms of the existence of certain complexes, which is our second main result.

\begin{thm}\label{main2}
Let $\mu: M \to X$ be an $R$-homomorphism such that $X$ is maximal Cohen-Macaulay. Then the following are equivalent.
\begin{enumerate}[\rm(1)]
\item $\mu$ is a special $\cm$-preenvelope of $M$.
\item There exists an $R$-complex $C=(0 \to C^{-1} \xrightarrow{d^{-1}} C^0 \xrightarrow{d^{0}} C^1 \xrightarrow{d^1} C^2 \xrightarrow{d^2} \cdots \xrightarrow{d^{d-3}} C^{d-2} \to 0)$ with $C^i$ free for $1 \le i \le d-2$ and $d^{-1}= \mu$ such that $\cod{\H^i(C)} >i+1$ for any $i$.
\end{enumerate}
\end{thm}

%Let $M$ be an $R$-module. Applying Theorem \cite[Theorem A]{H} repeatedly, we obtain an $R$-complex 
%$$
%0 \to M \xrightarrow{d^0} X^0 \xrightarrow{d^1} X^1 \xrightarrow{d^2} \cdots
%$$
%such that the induced morphisms $\mu^i: \cok d^{i-1}\to X^i$ are special \cm$-preenvelopes. Such a complex is called a special proper $\cm$-coresolution of $M$, and it is called a minimal proper $\cm$-coresolution if each $\mu^i$ is an $\cm$-cover. 
A special proper $\cm$-coresolution (resp. minimal proper $\cm$-coresolution) is such a complex that is built by taking special $\cm$-preenvelopes (resp. $\cm$-envelopes) repeatedly.
Using Theorem \ref{main1}, we prove the following result on the structure of special proper $\cm$-coresolutions as our third main result. 
\begin{thm}\label{main3}
Suppose that $R$ is Henselian. Let $M$ be an $R$-module and
$$
0 \to M \xrightarrow{d^{0}} X^0 \xrightarrow{d^1} X^1 \xrightarrow{d^2} \cdots
$$
a special proper $\cm$-coresolution of $M$.
Put $\mu^0:=d^0$ and let $\mu^i: \cok d^{i-1} \to X^i$ be the induced homomorphisms.
Then for each $i$ one has $\cod(\ker{\mu^i})> i$, and there exists an exact sequence 
$$
0 \to S^i \to \cok \mu^i \to T^i \to U^i \to 0 
$$
such that
	\begin{itemize}
	\item $\cod{S^i} > i+1$,
	\item $\cod{U^i} > i+2$, 
	\item $T^i$ satisfies $(S_2)$, 
	\item $(T^i)^\dagger$ has finite injective dimension and satisfies $(S_{i+3})$.
	\end{itemize}
\end{thm}
We should remark that our Theorem \ref{main1} guarantees that the converse of the statement of Theorem \ref{main3} also holds:
If a complex of $R$-modules
$$
0 \to M \xrightarrow{d^{0}} X^0 \xrightarrow{d^1} X^1 \xrightarrow{d^2} \cdots
$$
with each $X^i$ maximal Cohen-Macaulay satisfies the conditions in the conclusion of Theorem \ref{main3}, then this is a special proper $\cm$-coresolution of $M$.
Furthermore, it turns out that Theorem \ref{main3} recovers a main theorem of Holm \cite[Theorem C]{H} in the Henselian case.

\section{Preliminaries}
In this section, we give several basic definitions and remarks for later use. 
\begin{dfn}
Let $\X$ be a full subcategory of $\mod{R}$ and $M$ an $R$-module.
\begin{enumerate}[\rm(1)]
\item
Let $\pi: X \to M$ be an $R$-homomorphism such that $X \in \X$.
	\begin{enumerate}[\rm(a)]
	\item $\pi$ is called an {\it $\X$-precover} of $M$ if 
$$
\Hom_R(X', \pi):\Hom_R(X', X) \to	\Hom_R(X', M)
$$
 is an epimorphism for any $X' \in \X$.
	\item $\pi$ is called a {\it special $\X$-precover} of $M$ if $\pi$ is an $\X$-precover and satisfies $\Ext_R^1(\X, \ker\pi)= 0$. %\,(i.e. $\ker\pi$ has finite injective dimension). 
	\item $\pi$ is called an {\it $\X$-cover} of $M$ if $\pi$ is an $\X$-precover and for any $\phi \in \End_R(X)$, $\pi \phi= \pi$ implies $\phi$ is an automorphism. 
	\end{enumerate} 
\item
Let $\mu: M \to X$ be an $R$-homomorphism such that $X \in \X$.
\begin{enumerate}[\rm(a)]
	\item $\mu$ is called an {\it $\X$-preenvelope} of $M$ if 
$$
\Hom_R(\mu, X'):\Hom_R(X, X') \to \Hom_R(M, X')
$$
 is an epimorphism for any $X' \in \X$.
	\item $\mu$ is called a {\it special $\X$-preenvelope} of $M$ if $\mu$ is an $\X$-preenvelope and satisfies $\Ext_R^1(\cok\mu, \X)= 0$.
	\item $\mu$ is called an {\it $\X$-envelope} of $M$ if $\mu$ is an $\X$-preenvelope and for any $\phi \in \End_R(X)$, $\phi \mu= \mu$ implies $\phi$ is an automorphism. 
	\end{enumerate} 
\end{enumerate} 
\end{dfn}

\begin{rem}
\begin{enumerate}[\rm(1)]
\item
$\X$-precovers are not necessarily epimorphisms in general. If $\X$ contains $R$, then every $\X$-precover is an epimorphism. 
\item
Let $M$ be an $R$-module. An $\X$-cover of $M$ is unique in the following sense: for two $\X$-covers $\pi :X \to M$ and $\pi': X' \to M$ of $M$, there exists an isomorphism $\phi: X \to X'$ such that $\pi' \phi = \pi$.
Dually, an $\X$-envelope of $M$ is unique.
\item
By definition, a special $\X$-precover is an $\X$-precover.  If $\X$ is closed under extensions, then an $\X$-cover is a special $\X$-precover: this result is called Wakamatsu's lemma (see \cite{X}).
The statement where ``cover" is replaced with ``envelope" also holds true.
\item
Let $M$ be an $R$-module. If $\Ext_R^1(\cm, M)=0$, then $M$ has finite injective dimension since $\Ext_R^{d+1}(k, M)\cong \Ext_R^1(\syz^d k, M)=0$. Here, $\syz$ denotes the syzygy functor.
On the other hand, if $M$ has finite injective dimension, then $\Ext_R^1(\cm, M)=0$; see \cite{AB}.
Therefore, a special $\cm$-precover of $M$ is nothing but a maximal Cohen-Macaulay approximation of $M$.
\end{enumerate}
\end{rem}

 \begin{dfn}
 Let $M$ be an $R$-module, and
\begin{align*}
0 \to M \xrightarrow{d^0} X^0 \xrightarrow{d^1} X^1 \xrightarrow{d^2} \cdots \tag{*}
\end{align*}
be an $R$-complex with $X^i \in \cm$ for each $i$. Put $\mu^0:=d^0$, and let $\mu^i:\cok d^{i-1} \to X^i$ be the induced morphisms for $i \ge 1$.
If the $\mu^i$ are $\cm$-preenvelopes (resp. special $\cm$-preenvelopes; resp. $\cm$-envelopes), we call \eqref{***} a {\it proper $\cm$-coresolution} (resp. a {\it special proper $\cm$-coresolution}; resp. a {\it minimal proper $\cm$-coresolution}). 
By virtue of \cite[Theorem A]{H}, we can construct a special proper $\cm$-coresolution of $M$, and if $R$ is Henselian, we can construct a minimal proper $\cm$-coresolution of $M$ for any $M \in \mod{R}$.
 \end{dfn}
%\begin{rem}
%Let $M$ be a finitely generated $R$-module. If $\Ext_R^1(\cm, M)=0$, then $M$ has finite injective dimension since $\Ext_R^{d+1}(k, M)\cong \Ext_R^1(\syz^d k, M)=0$. Here, $\syz$ denotes the syzygy functor.
%On the other hand, if $M$ has finite injective dimension, then $\Ext_R^1(\cm, M)=0$; see \cite{AB}.
%Therefore, a special $\cm$-precover of $M$ is nothing but a maximal Cohen-Macaulay approximation of $M$.
%\end{rem}
 
\begin{dfn}
Let $M$ and $N$ be $R$-modules. 
We define $\rad_R(M, N)$ as the subgroup of $\Hom_R(M, N)$ consisting of homomorphisms $f$ which satisfy the following condition: there is no non-zero direct summand $M'$ of $M$ such that $M'$ is isomorphic to a direct summand of $N$ via $f$.
This definition is equivalent to the definition of $\rad_R$ in \cite{K} if $R$ is Henselian. 
\end{dfn}

\begin{rem}
Let $\X$ be a full subcategory of $\mod{R}$, and let $0 \to K \xrightarrow{f} X \xrightarrow{\pi} M$ be an exact sequence in $\mod{R}$, where $\pi$ is an $\X$-precover.
Suppose that $R$ is Henselian. Then $\pi$ is an $\X$-cover if and only if $f \in\rad_R(K, X)$.
For the proof, see \cite[Proposition 2.4]{R}.
\end{rem}
 
\section{The structure of special $\cm$-preenvelopes}
In this section, we prove our main result on the structure of $\cm$-(pre)envelopes.
We give several lemmas used in the proof of the main theorem.
The first one is due to Holm \cite[Lemma 3.2, Proposition 3.3]{H}.
\begin{lem}{\label{lem1}}%\cite[Lemma 3.2, Proposition 3.3]{H}
\begin{enumerate}[\rm(1)]
\item
Let $M$ be an $R$-module. Then $\Ext_R^1(M, \cm)=0$ if and only if $\Ext_R^1(M, \om)=0=\Ext_R^1(\cm, M^\dagger)$.
\item
If $\mu:M \to X$ is an $\cm$-preenvelope (resp. a special $\cm$-preenvelope; resp. an $\cm$-envelope), then $\mu^\dagger:X^\dagger \to M^\dagger$ is an $\cm$-precover (resp. a special $\cm$-precover; resp. an $\cm$-cover).
\end{enumerate}
\end{lem}
\begin{proof}
We only give a proof of the statement (2). For the proof of (1), see \cite[Lemma 3.2]{H}.

Suppose $\mu:M \to X$ is an $\cm$-preenvelope.
Let $f:Z \to M^\dagger$ be a homomorphism with $Z \in \cm$.
Since $\mu$ is an $\cm$-preenvelope, there exists a homomorphism $g:X \to Z^\dagger$
such that $f^\dagger \d_M = g \mu$.
Applying $(-)^\dagger$, one has $(\d_M)^\dagger f^{\dagger\dagger} = \mu^\dagger g^\dagger$.
Hence we obtain equalities
$$
\mu^\dagger g^\dagger \d_Z = (\d_M)^\dagger f^{\dagger\dagger} \d_Z = (\d_M)^\dagger \d_{M^\dagger} f = f ,
$$
where the last equality follows from \cite[Lemma 3.1]{H}.
This shows that $\mu^\dagger$ is an $\cm$-precover.
Next consider the case that $\mu$ is a special $\cm$-preenvelope. Then it follows from (1) that $\Ext_R^1(\cm, (\cok\mu)^\dagger)=0$.
One has $(\cok \mu)^\dagger \cong \ker (\mu^\dagger)$ because $(-)^\dagger:\mod{R} \to \mod{R}$ is a left exact functor.
Hence, $\mu^\dagger$ is a special $\cm$-precover.
Finally we suppose that $\mu$ is an $\cm$-envelope. Let $\phi$ be an endomorphism of $X^\dagger$ with $\mu^\dagger \phi = \mu^\dagger$. 
Then the commutative diagram
%Applying $(-)^\dagger$, we obtain $\phi^\dagger \mu^{\dagger\dagger}= \mu^{\dagger\dagger}$.
%$$
%\phi^\dagger \d_X \mu = \phi^\dagger \mu^{\dagger\dagger} \d_M =  \mu^{\dagger\dagger} \d_M = \d_X \mu
%$$
$$
\xymatrix
{
X \ar[r]_{\d_X}^{\cong} & X^{\dagger\dagger} \ar[dr]^{\phi^\dagger}   \\
M \ar[r]_{\d_M} \ar[u]_\mu \ar@/_23pt/[rr]_{\d_X \circ \mu} & M^{\dagger\dagger} \ar[u]_{\mu^{\dagger\dagger}} \ar[r]_{\mu^{\dagger\dagger}} & X^{\dagger\dagger}
}
$$
implies $\phi^\dagger$ is an automorphism and so is $\phi=(\d_{X^\dagger})^{-1} \phi^{\dagger\dagger} \d_{X^\dagger}$. Therefore, $\mu^\dagger$ is an $\cm$-cover.
%Because $\phi$ is an automorphism if and only if $\phi^\dagger$ is an automorphism.
\end{proof}

\begin{lem}\label{lem2}
Let $C$ be an $R$-module. Consider an exact sequence $0 \to L \xrightarrow{f} C \xrightarrow{g} N \to 0$ with $\cod L>0$ such that $N$ satisfies $(S_1)$.
Such an exact sequence is, if exists, unique up to isomorphisms of complexes with respect to $C$. 
\end{lem}
\begin{proof}
Let $0 \to L' \xrightarrow{f'} C \xrightarrow{g'} N' \to 0$ be a short exact sequence satisfying the same condition.
From \cite[Proposition 3.1]{AI}, $\d_{N}$ is a monomorphism.
Therefore, $N$ can be embedded in $\om^{\oplus r}$ for some integer $r$. 
Since $\ass{\Hom_R(L', \om)}=\supp{L'}\cap\ass{\om}=\emptyset$, we have $\Hom_R(L', \om)=0$ and thus, $\Hom_R(L', N)=0$.
Therefore, we have a commutative diagram
$$
\begin{CD}
0 @>>> L' @>f'>> C @>g'>> N' @>>> 0 \\
@. @VsVV  @| @VtVV @. \\
0 @>>> L @>f>> C @>g>> N @>>> 0.
\end{CD}
$$
Similarly, we have a commutative diagram
$$
\begin{CD}
0 @>>> L @>f>> C @>g>> N @>>> 0 \\
@. @Vs'VV  @| @Vt'VV @. \\
0 @>>> L' @>f'>> C @>g'>> N' @>>> 0.
\end{CD}
$$
These two commutative diagrams yield $f's's=f'$ and since $f'$ is a monomorphism, we conclude $s's=1$. 
Using the same argument, one has $ss'=1$. Thus, $s$ is an isomorphism and so is $t$.
\end{proof}

\begin{lem}\label{lem3}
Let $0 \to U \to L \xrightarrow{\alpha} M \to V \to 0$ be an exact sequence such that $\cod U \ge i$ and $\cod V \ge i+1$ for an integer $i$.
Then, we get an isomorphism $\Ext_R^l(M, \om) \cong \Ext_R^l(L, \om)$ for any integer $l<i$.
\end{lem}
\begin{proof}
From the short exact sequence $0 \to U \to L \to \im \alpha \to 0$,
% considering the long exact sequence of Ext, 
we have an exact sequence
$$
\Ext_R^{l-1}(U, \om) \to \Ext_R^{l}(\im \alpha, \om) \to \Ext_R^l(L, \om) \to \Ext^l(U, \om).
$$
Since $\dim{U} < d-l$, using the local duality theorem, we get $\Ext_R^l(U, \om) \cong \H_{\m}^{d-l}(U)^{\vee}=0$ for $l<i$, where $(-)^{\vee}$ stands for the Matlis dual. Therefore, $\Ext_R^{l}(\im \alpha, \om) \cong \Ext_R^l(L, \om)$ for $l<i$.
Similarly, because we have $\Ext_R^l(V, \om)=0$ for $l<i+1$ by the local duality theorem, using the same argument for $0 \to \im \alpha \to M \to V \to 0$, we get an isomorphism $\Ext_R^{l}(M, \om) \cong \Ext_R^l(\im \alpha, \om)$ for $l<i+1$.
Consequently, we obtain an isomorphism $\Ext_R^l(M, \om) \cong \Ext_R^l(L, \om)$ for $l<i$.
\end{proof}

Let $M$ be an $R$-module and $P_1 \xrightarrow{f} P_0 \to M \to 0$ a projective presentation of $M$. Then we put $\tr_\om M := \cok(f^\dagger:P_0^\dagger \to P_1^\dagger)$ and call it an $\om$-transpose of $M$.
The following lemma is well-known in the case where $\omega=R$.
The proof of this lemma is shown along the same lines as in that of \cite[Proposition 2.6]{ABr}.
\begin{lem}\label{lem4}
Let $M$ be an $R$-module. Then there exist isomorphisms
\begin{enumerate}[\rm(1)]
\item $\ker \d_M \cong \Ext_R^1(\tr_\om M, \om)$ and
\item $\cok \d_M \cong \Ext_R^2(\tr_\om M, \om)$.
\end{enumerate}
\end{lem}

We are now ready to show our first main result.

\begin{proof}[Proof of Theorem \ref{main1}]
(a)$\Rightarrow$(b): By Lemma \ref{lem1}, $\mu^\dagger:X^\dagger \to M^\dagger$ is a special $\cm$-precover. In particular, $\mu^\dagger$ is an epimorphism.
Since $\d_X \mu = \mu^{\dagger\dagger} \d_M$ and $\mu^{\dagger\dagger}$ is a monomorphism, $\ker \d_M \cong \ker \mu$.
Using Lemma \ref{lem4}, one has $\ker \mu \cong \Ext_R^1(\tr_\om M, \om)$. 
Because $\om_{\p}$ is an injective $R_{\p}$-module for any minimal prime ideal $\p$, we conclude $\cod (\ker \mu)>0$.

Next, consider the case (2).
Then $\mu^\dagger$ is an $\cm$-cover, and hence there exists an exact sequence $0 \to Y \xrightarrow{f} X^\dagger \xrightarrow{\mu^\dagger} M^\dagger \to 0$ such that $Y$ has finite injective dimension and $f \in \rad_R(Y, X^\dagger)$.
Applying $(-)^\dagger$, we get an exact sequence 
$$
0 \to M^{\dagger\dagger} \xrightarrow{\mu^{\dagger\dagger}} X^{\dagger\dagger} \xrightarrow{f^\dagger} Y^\dagger \to \Ext_R^1(M^\dagger, \om) \to 0.
$$
%\xymatrix{
%0 \ar[r] & M^{\dagger\dagger} \ar[r] & X^{\dagger\dagger} \ar[rd] \ar[rr] & & Y^\dagger \ar[r] & \Ext_R^1(M^\dagger, \om) \ar[r] & 0 \\
%& & & \cok \mu^{\dagger\dagger} \ar[ru] \ar[rd] & & & \\
%& & 0 \ar[ru] & & 0 & &
%}
Then $f^\dagger$ can be decomposed into an epimorphism $g:X^{\dagger\dagger} \to \cok (\mu^{\dagger\dagger})$ and a monomorphism $h:\cok (\mu^{\dagger\dagger}) \to Y^\dagger$.
Since $\depth_{R_\p}{(M^\dagger)_\p} \ge \mn \{\h \p, 2 \}$, $M^\dagger$ satisfies $(S_2)$, and therefore we get $\cod (\Ext_R^1(M^\dagger, \om))> 2$.
Hence $h^\dagger$ is an isomorphism by Lemma \ref{lem3}.
By the depth lemma, $Y$ satisfies $(S_3)$ and hence \cite[Proposition 3.1]{AI} implies that $\d_Y$ is an isomorphism.
Therefore, $f\in \rad_R(Y, X^\dagger)$ yields $f^{\dagger\dagger}=g^\dagger h^\dagger \in \rad_R(Y^{\dagger\dagger}, X^{\dagger\dagger\dagger})$, and hence $g^\dagger \in \rad_R((\cok(\mu^{\dagger\dagger}))^\dagger, X^{\dagger\dagger\dagger})$.
Because $g^\dagger \in \rad_R((\cok(\mu^{\dagger\dagger}))^\dagger, X^{\dagger\dagger\dagger})$, $\cok(\mu^{\dagger\dagger})$ and $X^{\dagger\dagger\dagger}$ has no non-zero common direct summand via $g^\dagger$.
This shows that $X^{\dagger\dagger}$ and $\cok (\mu^{\dagger\dagger})$ has no non-zero common direct summand via $g$.
Consequently, $\cok(\mu^{\dagger\dagger})$ has no non-zero free summand.
%$$
%f \in \rad(Y, X^\dagger) &\Longrightarrow f^{\dagger\dagger} \in \rad(Y^{\dagger\dagger}, X^{\dagger\dagger\dagger})
%& \Longrightarrow g^\dagger \in \rad((\cok\mu^{\dagger\dagger})^\dagger, X^{\dagger\dagger\dagger})
%& \Longrightarrow g^\dagger \in \rad(X^{\dagger\dagger}, (\cok\mu^{\dagger\dagger}))
%& \Longrightarrow \cok\mu^{\dagger\dagger} has no free summand.
%$$

(b)$\Rightarrow$(c): By the local duality theorem, we have $(\ker\mu)^\dagger=0$.
Since $(\ker\mu)^\dagger=0=\Ext_R^1(\cok\mu, \om)$, $\mu^\dagger$ is an epimorphism.
Taking a short exact sequence $0 \to Y \to X^\dagger \xrightarrow{\mu^\dagger} M^\dagger \to 0$ and applying $(-)^\dagger$ to this sequence, we obtain an exact sequence
$$
0 \to M^{\dagger\dagger} \xrightarrow{\mu^{\dagger\dagger}} X^{\dagger\dagger} \to Y^\dagger \to \Ext_R^1(M^\dagger, \om) \to 0.
$$
%\xymatrix{
%0 \ar[r] & M^{\dagger\dagger} \ar[r] & X^{\dagger\dagger} \ar[rd] \ar[rr] & & Y^\dagger \ar[r] & \Ext_R^1(M^\dagger, \om) \ar[r] & 0 \\
%& & & \cok \mu^{\dagger\dagger} \ar[ru] \ar[rd] & & & \\
%& & 0 \ar[ru] & & 0 & &
%}.
Because  $\d_X \mu = \mu^{\dagger\dagger} \d_M$, we get another exact sequence
$$
\ker (\mu^{\dagger\dagger}) =0 \to \cok\d_M \to \cok \mu \to \cok (\mu^{\dagger\dagger}) \to 0 .
$$
Combining this two sequences, we have an exact sequence
$$
0 \to \cok\d_M \to \cok \mu \xrightarrow{\alpha} Y^\dagger \to \Ext_R^1(M^\dagger, \om) \to 0.
$$
We verify that this sequence satisfies the condition (c).

Since $\om_\p$ has injective dimension at most 1 for any prime ideal $\p$ with $\h \p \le 1$, $\cok \d_M \cong \Ext_R^2(\tr_\om, \om)$ has codimension at least 2.
As $M^\dagger$ satisfies $(S_2)$, $\cod (\Ext_R^1(M^\dagger, \om))>2$.
Because $Y$ satisfies $(S_3)$, $\d_Y$ is an isomorphism.
Hence $(Y^\dagger)^\dagger \cong Y$ satisfies $(S_3)$.
By Lemma \ref{lem1}(1), one has $\Ext_R^1(\cm, (\cok\mu)^\dagger)=0$, 
that is, $(Y^\dagger)^\dagger \cong (\cok \mu)^\dagger$ has finite injective dimension.
Consequently, $\mu$ satisfies the conditions in $(c)$.
Moreover, since $\im \alpha \cong \cok (\mu^{\dagger\dagger})$, the implication (b) $\Rightarrow$ (c) also holds in the case of (2).

(c) $\Rightarrow$ (a): First, we prove $\Ext_R^1(\cok\mu, \cm)=0$.
By assumption, there exists an exact sequence
$$
0 \to S \to \cok{\mu} \xrightarrow{\alpha} T \to U \to 0
$$
which satisfies the conditions in (c).
By the local duality theorem, we have $S^\dagger=0=\Ext_R^1(S, \om)$, hence $\Ext_R^1(S, \cm)=0$. Therefore we have only to prove $\Ext_R^1(\im \alpha, \cm)=0$.
Since $(\im\alpha)^\dagger \cong T^\dagger$ has finite injective dimension, one has $\Ext_R^1(\cm, (\im \mu)^\dagger)=0$.
The short exact sequence $0 \to \im \alpha \to T \to U \to 0$ induces an exact sequence
$$
\Ext_R^1(T, \om) \to \Ext_R^1(\im \alpha, \om) \to \Ext_R^2(U, \om).
$$
Since $\cod U >2$, $\Ext_R^2(U, \om)=0$. 
From \cite[Proposition 3.1]{AI}, we get $\Ext_R^1(T^{\dagger\dagger}, \om)=0$ and $T \cong T^{\dagger\dagger}$. Therefore, $\Ext_R^1(\im \alpha, \om)=0$.
Using Lemma \ref{lem1}, we conclude $\Ext_R^1(\im \alpha, \cm)=0$.

Next, we show that $\mu$ is a special $\cm$-preenvelope.
Let $Z$ be a maximal Cohen-Macaulay $R$-module and $f:M \to Z$ an $R$-homomorphism. Since $\ass{\Hom_R(\ker \mu, Z)}=\supp(\ker \mu) \cap \ass{Z}= \emptyset$, $f$ can be lifted to $\im \mu$. Furthermore, because $\Ext_R^1(\cok \mu, Z)=0$, the morphism $\im \mu \to Z$ can be lifted to $X$.
This shows that $\mu$ is an $\cm$-preenvelope.

Consider the case (2).
Using \cite[Theorem A]{H}, there exists an $\cm$-envelope $\mu':M \to X'$.
We show that $\mu$ is isomorphic to $\mu'$.
As we have already seen, $\mu$ is a special $\cm$-preenvelope. Hence there exists a commutative diagram
$$
\begin{CD}
M @>\mu'>> X' @>>> \cok \mu' @>>> 0 \\
@| @VfVV @V\overline{f}VV \\
M @>\mu>> X @>>> \cok \mu @>>> 0 \\
@| @VgVV @V\overline{g}VV \\
M @>\mu'>> X' @>>> \cok \mu' @>>> 0 
\end{CD}
$$
and since $\mu'$ is an $\cm$-envelope, $g f$ and $\overline{g} \overline{f}$ are automorphisms.
Consider the chain map:
$$
\begin{CD}
F @.\ =\ @.(@. 0 @>>> M @>\mu'>> X' @>>> \cok \mu' @>>> 0 @.) \\
@V{u}VV @.@. @VVV @| @VfVV @V\overline{f}VV @VVV \\
G @.\ =\ @.(@. 0 @>>> M @>\mu>> X @>>> \cok \mu @>>> 0 @.) .
\end{CD}
$$
Then the cokernel of $u$ is of the form $\cok u = (0 \to 0 \to \cok f \to \cok \overline{f} \to 0)$ and by calculating the homologies of $\cok u$, we conclude $\cok f \cong \cok \overline{f}$.
Since $f$ and $f'$ are split monomorphisms, $\cok f \cong \cok \overline{f}$ is a maximal Cohen-Macaulay and $(\cok f)^\dagger \cong (\cok \overline{f})^\dagger$ has finite injective dimension by Lemma \ref{lem1}(1). 
This shows that $\cok f \cong \cok \overline{f}$ is a free $R$-module.
From the assumption, there exists an exact sequence 
$$
0 \to S \to \cok \mu \xrightarrow{\alpha} T \to U \to 0 
$$
which satisfies the conditions in (c).
On the other hand, as we have shown in the proof of the implications (a)$\Rightarrow$(b)$\Rightarrow$(c), $\cok \mu'$ also admits an exact sequence 
$$
0 \to S' \to \cok \mu' \xrightarrow{\alpha'} T' \to U' \to 0 
$$
which satisfies the conditions in (c).
Then we obtain two short exact sequences 
\begin{align*}
0 \to S \to \cok \mu \to \im \alpha \to 0, \tag{1} \\
0 \to S' \to \cok \mu' \to \im {\alpha'} \to 0. \tag{2}
\end{align*}

Since $\cok \mu \cong \cok \mu' \oplus \cok \overline{f}$,  we obtain a short exact sequence \begin{align*}
0 \to S' \to \cok \mu \to \im {\alpha'} \oplus \cok \overline{f} \to 0 \tag{2'}
\end{align*}
from (2) by taking the direct sum with $\cok \overline{f}$.
By assumption, $S$ and $S'$ have codimension at least 2. 
Since $T$ and $T'$ satisfy $(S_1)$, those submodules $\im {\alpha}$ and $\im {\alpha'}$ are also satisfy $(S_1)$.  
Using Lemma \ref{lem2}, (1) and (2') are isomorphic as complexes. In particular, $ \im{\alpha'} \oplus \cok \overline{f}$ is isomorphic to $\im \alpha$.
Because $\im \alpha$ has no non-zero free summand, the free module $\cok \overline{f}$ is $0$.
Consequently, $\mu$ is isomorphic to $\mu'$.
\end{proof}
Let us construct examples of special $\cm$-preenvelopes and $\cm$-envelopes by using Theorem \ref{main1}.
\begin{ex}
\begin{enumerate}[\rm(1)]
\item Let $R$ be a $d$-dimensional Cohen-Macaulay local ring with canonical module $\om$ and $\underline{x}=x_1, x_2, \ldots x_n$ an $R$-regular sequence with $n \ge 3$.
Consider an exact sequence
$$
0 \to M \xrightarrow{\mu} R^{\oplus n} \xrightarrow{(x_1, \ldots, x_n)} R \to R/(\underline{x}) \to 0.
$$
Since the cokernel of $\mu$ admits an exact sequence $0 \to \cok \mu \to R \to R/(x) \to 0$ which satisfies the conditions in (c) in Theorem \ref{main1}, $\mu: M \to R^{\oplus n}$ is a special $\cm$-preenvelope.
\item Let $R$ be a $d$-dimensional Cohen-Macaulay local domain with canonical module $\om$ such that $d \ge 2$.
Take an $R$-sequence $x, y$ and a non-zero element $f \in R$.
Consider a commutative diagram
$$
\begin{CD}
0 @>>> R @>\binom{xf}{yf}>> R^{\oplus 2} @>>> M @>>> 0 \\
@. @VVfV @| @VVgV @. \\
0 @>>> R @>\binom{x}{y}>> R^{\oplus 2} @>>> (x,y) @>>> 0.
\end{CD}
$$
Using the snake lemma, we have a short exact sequence $0 \to R/(f) \to M \xrightarrow{g} (x,y) \to 0$.
Let $\mu:M \to R$ be the composition of $g:M \to (x, y)$ and the inclusion $(x, y) \to R$.
Then $\ker \mu \cong \ker g \cong R/(f)$ has codimension at least $1$, and $\cok \mu \cong R/(x, y)$ has codimension at least 2 and $R/(x, y)$ has no free summand.  Therefore, $\mu$ is an $\cm$-envelope.   
\end{enumerate}
\end{ex}

%In the last of this section, we give another characterization of special $\cm$-preenvelopes.

\begin{rem}
For each $Y \in \mod{R}$, $Y$ has finite injective dimension if and only if $Y$ admits an exact sequence $0 \to W_n \to W_{n-1} \to \cdots \to W_0 \to Y \to 0$ with $W_i \in \add{\om}$ and $n = d- \depth_{R}{Y}$.
Therefore, for an $R$-homomorphism $\pi: X \to M$ with $X \in \cm$, the following are equivalent.
\begin{enumerate}[(1)]
\item $\pi$ is a special $\cm$-precover.
\item There exists an exact sequence $0 \to W_{n} \to W_{n-1} \to \cdots \to W_1 \to X \xrightarrow{\pi} M \to 0$ with $W_i \in \add{\om}$ and $n=d-\depth_{R}{M}$.
\end{enumerate}
%Therefore, we can consider that Theorem \ref{main2} is the dual of this result.
\end{rem}
Now, let us give a proof of Theorem \ref{main2}
\begin{proof}[Proof of Theorem \ref{main2}]
(1) $\Rightarrow$ (2):
Because of Lemma \ref{lem1}, $\mu^\dagger:X^\dagger \to M^\dagger$ is a special $\cm$-precover and there exists an exact sequence
\begin{align*}\label{***}
0 \to W_{d-2} \xrightarrow{f_{d-2}} W_{d-3} \xrightarrow{f_{d-3}} \cdots \xrightarrow{f_2}  W_1 \xrightarrow{f_1} X^\dagger \xrightarrow{\mu^\dagger} M^\dagger \to 0 \tag{*}
\end{align*}
such that $W_i \in \add{\om}$. 
Applying $(-)^\dagger$, we obtain an $R$-complex 
$$
C=(0 \to M \xrightarrow{\mu= \d_X^{-1} \mu^{\dagger\dagger} \d_M} X \xrightarrow{f_1^\dagger \d_X} W_1^\dagger \xrightarrow{f_2^\dagger} \cdots \xrightarrow{f_{d-2}^\dagger} W_{d-2}^\dagger \to 0).
$$
Decompose \eqref{***} into short exact sequences $0 \to U_i \xrightarrow{u_i} W_i \xrightarrow{v_i} U_{i-1} \to 0\,\,(0\le i \le d-2)$ where $W_0 := X^\dagger$ and $U_{-1} := M^\dagger$. Applying $(-)^\dagger$ to these sequences, we obtain exact sequences 
$0 \to U_{i-1}^\dagger \xrightarrow{v_i^\dagger} W_i^\dagger \xrightarrow{u_i^\dagger} U_i^\dagger \to \Ext_R^1(U_{i-1}, \om) \to 0$. 
Since $v_i^\dagger$ are monomorphisms, 
$$
\H^i(C) = \ker{f_{i+1}^\dagger}/\im{f_i^\dagger}
= \ker{u_i^\dagger}/\im{f_i^\dagger} 
\cong \cok{u_{i-1}^\dagger} 
\cong \Ext_R^1(U_{i-2}^\dagger, \om)
$$
for $1 \le i \le d-2$.
From the exact sequence \eqref{***}, $(U_{i-2})_\p$ is a maximal Cohen-Macaulay $R_\p$-module for any prime ideal $\p$ with $\h{\p} \le i+1$ by the depth lemma. Hence $\cod \H^i(C) >i+1$ for $i \ge1$. 
For $i=-1, 0$, 
\begin{align*}
\H^{-1}(C) &= \ker(\mu^{\dagger\dagger}\d_M) 
= \ker \d_M 
\cong \Ext_R^1(\tr_\om M, \om) \\
\H^{0}(C) &= \ker f_1^\dagger / \im(\mu^{\dagger\dagger} \d_M) 
\cong \cok{\d_M} 
\cong \Ext_R^2(\tr_\om M, \om).
\end{align*}
Therefore, $\cod \H^i(C) > i+1$ for any $i$.

(2) $\Rightarrow$ (1):
First note that if $d \le 2$, then this implication holds.
Indeed, since $\cod(\ker \mu)>0$ and $\cod(\cok \mu)>1$, $\mu$ is a special $\cm$-preenvelope by Theorem \ref{main1}.
 
Next, consider the case $d\ge 3$.
Put $d^0:=\mu$ and denote by $\mu^i$ the induced homomorphism $\cok{d^{i-1}} \to C^{i}$ for each $1\le i \le d-2$.
Note that there exists an exact sequence
$$
0 \to \H^{i-1}(C) \to \cok{d^{i-1}} \xrightarrow{\mu^{i}} C^{i} \to \cok d^i \to 0.
$$
Let us show $\Ext_R^1(\cok d^i, \cm)=0$ for $0 \le i \le d-2$.
By assumption, $\cok d^{d-2} \cong \H^{d-2}(C)$ and $\H^{d-3}(C)$ have codimension at least 2. Hence $\Ext_R^1(\cok d^{d-2}, \cm) =0= \Ext_R^1(\H^{d-3}(C), \cm)$.
From the exact sequence 
$$
0 \to \im \mu^{d-2} \to C^{d-2} \to \cok d^{d-2} \to 0,
$$
we have $\Ext_R^1(\im \mu^{d-2}, \cm)=0$.
Hence one has $\Ext_R^1(\cok d^{d-3}, \cm)=0$ from the exact sequence 
$$
0 \to \H^{d-3}(C) \to \cok d^{d-3} \to \im \mu^{d-2} \to 0.
$$
Iterating this procedure, we get $\Ext_R^1(\cok d^i, \cm)=0$ for $0 \le i \le d-2$.
Since $\Ext_R^1(\cok \mu, \cm)=0$ and $\ker \mu \cong \H^{-1}(C)$ has positive codimension, $\mu$ is a special $\cm$-preenvelope. 
\end{proof} 

\section{The structure of special proper $\cm$-coresolutions}
In this section, we study the structure of special proper $\cm$-coresolutions by using Theorem \ref{main1}.
%Let $M$ be an $R$-module, and
%\begin{align*}
%0 \to M \xrightarrow{d^0} X^0 \xrightarrow{d^1} X^1 \xrightarrow{d^2} \cdots \tag{*}
%\end{align*}
%be an $R$-complex with $X^i \in \cm$ for each $i$. Put $\mu^0:=d^0$, and let $\mu^i:\cok d^{i-1} \to X^i$ be the induced morphisms for $i \ge 1$.
%If the $\mu^i$ are $\cm$-preenvelopes (resp. special $\cm$-preenvelopes; resp. $\cm$-envelopes), we call \eqref{***} a \it{proper $\cm$-coresolution} (resp. a \it{special proper $\cm$-coresolution}; resp. a \it{minimal proper $\cm$-coresolution}). 
%By virtue of Theorem \ref{holm}, we can construct a special proper $\cm$-coresolution of $M$, and if $R$ is Henselian, we can construct a minimal proper $\cm$-coresolution of $M$ for any $M \in \mod{R}$.
From now on, we assume that $R$ is Henselian. 
The following lemma is the key to prove our last theorem.
\begin{lem}
Let  
$$
0 \to M \xrightarrow{d^0} X^0 \xrightarrow{d^1} X^1 \xrightarrow{d^2} \cdots
$$
be a special proper $\cm$-coresolution of an $R$-module $M$.
Put $\mu^0=d^0$, and let $\mu^i: \cok d^{i-1} \to X^i$ be the induced morphisms for $i \ge 1$. Then the following holds:
\begin{enumerate}[\rm(1)]
\item $\cok \mu^i$ are unique up to free summands with respect to $M$ and
\item $\ker \mu^i$ are unique up to isomorphisms with respect to $M$.
\end{enumerate}
\end{lem}

\begin{proof}
Set $M^i :=\cok d^{i-1}\,\,(i \ge 1)$ and $M^0:=M$.

$(1)$: As we saw in  the proof of the implication (c)$\Rightarrow$(a) in Theorem \ref{main1}(2), for any $R$-module $M$, the cokernel of a special $\cm$-preenvelope of $M$ is unique up to free summands.
On the other hand, for a special $\cm$-preenvelope $\mu:M \to X$ and a free module $F$, 
$\mu \oplus F: M \oplus F \to X \oplus F$ is also a special $\cm$-preenvelope since special $\cm$-preenvelopes are characterized only by their kernels and cokernels by Theorem \ref{main1}. 
Consequently, $\cok \mu^i$ are unique up to free summands.

$(2)$: Let $M$ be an $R$-module and $F$ a free module.
Then the kernels of special $\cm$-preenvelopes of $M$ and $M \oplus F$ are isomorphic by using Lemma \ref{lem3}.
This shows $(2)$.
\end{proof}
Now, let us prove Theorem \ref{main3} given in the introduction.
\begin{proof}[Proof of Theorem \ref{main3}]
To prove this theorem, we have only to construct a such special proper $\cm$-coresolution.

Take a special $\cm$-preenvelope $\mu^0:M \to X^0$. 
Then $\ker \mu^0$ and $\cok \mu^0$ satisfy the desired conditions by Theorem \ref{main1}. 
For $i>1$, assume that there exists an $R$-complex 
$$
0 \to M \xrightarrow{d^0} X^0 \xrightarrow{d^1} X^1 \xrightarrow{d^2} \cdots \xrightarrow{d^{i-1}} X^{i-1}
$$
which satisfies the desired conditions for $j < i$.
By assumption, there exists an exact sequence 
$$
0 \to S^{i-1} \to \cok \mu^{i-1} \xrightarrow{u} T^{i-1} \to U^{i-1} \to 0 
$$
which satisfies the conditions in the statement.
Since $(T^{i-1})^\dagger$ has finite injective dimension, we can choose a short exact sequence $0 \to A \to W \to (T^{i-1})^\dagger \to 0$ where $W \in \add{\om}$ and $A$ has finite injective dimension.
Applying $(-)^\dagger$, we have an exact sequence
$$
0 \to T^{i-1} \xrightarrow{v} W^\dagger \to A^\dagger \to \Ext_R^1((T^{i-1})^\dagger, \om) \to 0.
$$
Set $\mu^i=vu:\cok d^{i-1} \to W^\dagger$ and we show it satisfies the conditions in the statement.
Then there exists an exact sequence
$$
\ker v =0 \to \cok u \to \cok \mu^i \to \cok v \to 0.
$$
Splicing this sequence with $0 \to \cok v \to A^\dagger \to \Ext_R^1((T^{i-1})^\dagger, \om) \to 0$, we obtain an exact sequence
$$
0 \to \cok u \to \cok \mu^i \to A^\dagger \to \Ext_R^1((T^{i-1})^\dagger, \om) \to 0.
$$
Set $S^i =\cok u=U^{i-1}$, $T^i =A^\dagger$ and $U^i =\Ext_R^1((T^{i-1})^\dagger, \om)$.
By assumption, $\ker \mu^i \cong \ker u \cong S^{i-1}$ has codimension at least $i+1$ and $S^i =\cok u = U^{i-1}$ has codimension at least $i+2$.
Since $(T^{i-1})^\dagger$ satisfies $(S_{i+2})$, $((T^{i-1})^\dagger)_\p$ is a maximal Cohen-Macaulay $R_\p$-module for any prime ideal $\p$ with $\h \p \le i+2$.
Hence $U^i =\Ext_R^1((T^{i-1})^\dagger, \om)$ has codimension at least $i+3$.
Note that $T^i =A^\dagger$ satisfies $(S_2)$.
Because $A$ satisfies $(S_{i+3})$, $\d_A$ is an isomorphism.
Therefore, $(T^i)^\dagger =A^{\dagger\dagger} \cong A$ has finite injective dimension and satisfies $(S_{i+3})$. 
By induction on $i$, the proof of theorem is completed.
\end{proof}

The following result is shown in \cite[Theorem C]{H} even if $R$ is not Henselian. It is also shown by examining the structure of an $\cm$-envelope concretely.
\begin{cor}
For any $R$-module $M$, the minimal proper $\cm$-coresolution of $M$ has length at most $d$.
\end{cor}
\begin{proof}
For an $R$-module $M$, take a minimal proper $\cm$-coresolution
$$
0 \to M \xrightarrow{\mu^{0}} X^0 \xrightarrow{d^1} X^1 \xrightarrow{d^2} \cdots.
$$
Let $\mu^i: \cok d^{i-1} \to X^i$ be the induced morphism.
For $i=d-1$, there exists an exact sequence
$$
0 \to S^{d-1} \to \cok d^{d-1} \to T^{d-1} \to U^{d-1} \to 0 
$$
which satisfies the conditions in the statement of Theorem \ref{main1}.
Since $S^{d-1}$ and $U^{d-1}$ have codimension at least $d+1$, $\cok d^{d-1}$ is isomorphic to $T^{d-1}$.
On the other hand, $(T^{d-1})^\dagger$ satisfies $(S_{d+2})$ and has finite injective dimension, i.e., $(T^{d-1})^\dagger \in \add{\om}$.
As $T^{d-1}$ satisfies $(S_2)$, $T^{d-1}$ is isomorphic to $(T^{d-1})^{\dagger\dagger}$, whence, free.
Consequently, the minimal proper $\cm$-coresolution ends in $X^d$.
\end{proof}

\section*{Acknowledgments}
The author is grateful to his supervisor Ryo Takahashi for a lot of comments, suggestions and discussions.


\begin{thebibliography}{99}
  \bibitem{AI} T. Araya and K.-i. Iima,  Locally Gorensteinness over Cohen-Macaulay rings, \texttt{arXiv:1408.3796v1}.
\bibitem{ABr} M. Auslander and M. Bridger, Stable module theory, {\em Mem. Amer. Math. Soc.} No. 94, {\em American Mathematical Society, Providence, R.I.}, 1969.
  \bibitem{AB} M. Auslander and R.-O. Buchweitz, The homological theory of maximal Cohen-Macaulay approximations, Colloque en l'honneur de Pierre Samuel (Orsay, 1987), {\em M\'{e}m. Soc. Math. France (N.S.)}, No. 38 (1989), 5--37.
    \bibitem{AR} M. Auslander and I. Reiten, Applications of contravariantly finite subcategories, {\it Adv. Math.} {\bf{86}} (1991), 111--152.
    \bibitem{AS1} M. Auslander and S. O. Smal$\o$, Almost split sequences in subcategories, {\it J. Algebra} {\bf{69}} (1980), no. 2, 426--454.
     \bibitem{ASS} M. Auslander, S. O. Smal$\o$, Preprojective modules over Artin algebras, {\it J. Algebra} {\bf{66}} (1) (1980), 61--122.
    % \bibitem{AS2} M. Auslander and Smalo , Categorical methods is the representation theory of Artin rings, {\it Adv. Math.} {\bf{86}} (1991), 111--152.
    \bibitem{BH}W. Bruns, J. Herzog, Cohen-Macaulay rings, Cambridge University Press, 1998.        
     \bibitem{E} E. E. Enochs, Injective and flat covers, envelopes and resolvents, {\it Israel J. Math.} {\bf{39}} (3) (1981), 189--209.
     \bibitem{EJ} E. E. Enochs and O.M.G. Jenda, Gorenstein injective and projective modules, {\it Math. Z.} {\bf{220}} (4) (1995), 611--633.
     %\bibitem{EJ} E.E. Enochs and O.M.G. Jenda, Gorenstein injective and projective modules, {\it Math. Z.} {\bf{220}} (4) (1995), 611--633.                                                                       
  \bibitem{H} H. Holm, Approximations by maximal Cohen-Macaulay modules, {\it{Pacific J. Math}} (to appear).
  \bibitem{K} G. M. Kelly, On the radical of a category, {\it J. Austral. Math. Soc.} {\bf{4}} (1964), 299--307.
  %\bibitem{KS}H. Krause and G. Stevenson, A note on thick subcategories of stable derived categories. {\it Nagoya Math. J.} {\bf{212}} (2013), 87--96.
  \bibitem{R} R. Takahashi, On the category of modules of Gorenstein dimension zero, {\it Math. Z.} {\bf{251}} (2005), no. 2, 249--256.
  \bibitem{R2} R. Takahashi, A new approximation theory which unifies spherical and Cohen-Macaulay approximations, {\it Journal of Pure and Applied Algebra} {\bf{208}} (2007), 617--634.
   \bibitem{R3} R. Takahashi, Contravariantly finite resolving subcategories over commutative rings, {\it Amer. J. Math.} {\bf{133}} (2011), no. 2, 417--436. 
  \bibitem{X} J. Xu, Flat covers of modules, Lecture Notes in Mathematics 1643, {\it Springer-Verlag, Berlin}, (1996).
 \bibitem{Y} Y. Yoshino, Cohen-Macaulay approximations (Japanese), {\it Proceedings of the 4th Symposium on Representation
Theory of Algebras, Izu, Japan}, (1993), pp. 119--138.
\end{thebibliography}
\end{document}